\documentclass[12pt]{amsart}
\usepackage[margin=1.25in]{geometry}
\usepackage{graphicx}
\usepackage{amssymb}
\usepackage{enumerate}
\usepackage{epstopdf}
\usepackage{amsmath,amsthm,amscd,amssymb}
\usepackage{latexsym}
\usepackage[colorlinks,citecolor=red,pagebackref,hypertexnames=false]{hyperref}
\numberwithin{equation}{section}
\newtheorem{theorem}{Theorem}[section]
\newtheorem{corollary}[theorem]{Corollary}

\newtheorem{lemma}[theorem]{Lemma}

\newtheorem{remark}[theorem]{Remark}

\newtheorem{example}{Example}
\newtheorem{definition}[theorem]{Definition}

\begin{document}

\title[Random gap processes and asymptotically complete sequences]{Random gap processes and asymptotically complete sequences}

\author[E. Crossen Brown]{Erin Crossen Brown}
\address{E. Crossen Brown, Department of Mathematics, 
University of Rochester,
500 Joseph C. Wilson Blvd., Rochester, NY 14627}
\email{ecrossen@ur.rochester.edu}

\author[S. Mkrtchyan]{Sevak Mkrtchyan}
\address{S. Mkrtchyan, Department of Mathematics, 
University of Rochester,
500 Joseph C. Wilson Blvd., Rochester, NY 14627}
\email{sevak.mkrtchyan@rochester.edu}

\author[J. Pakianathan]{Jonathan Pakianathan}
\address{J. Pakianathan, Department of Mathematics, 
University of Rochester,
500 Joseph C. Wilson Blvd., Rochester, NY 14627}
\email{jonathan.pakianathan@rochester.edu}

\date{}

\thanks{
The second listed author was partially supported by the Simons Foundation Collaboration Grant No. 422190.}

\begin{abstract}
We study a process of generating random positive integer weight sequences $\{ W_n \}$ where the gaps between the weights $\{ X_n = W_n - W_{n-1} \}$ are i.i.d. positive integer-valued random 
variables. We show that as long as the gap distribution has finite $\frac{1}{2}$-moment, almost surely, the resulting weight sequence is asymptotically complete, i.e., all large enough multiples of the gcd of the possible gap values can be written as a sum of distinct weights. We then show a much stronger result that if the gap distribution has a moment generating function with large enough radius of convergence, then every large enough multiple of the gcd of gap values can be written as a sum of $m$ distinct weights for any fixed $m \geq 2$.



\end{abstract}

\maketitle

\setcounter{tocdepth}{1}
\tableofcontents
\setcounter{tocdepth}{3}

\section{Introduction}

This paper studies certain random gap processes and shows that with probability $1$ they generate well-distributed, asymptotically $m$-complete sequences for any $m \geq 2$. This directly relates to issues in number theory and topology, as we will elaborate below.

A {\bf complete sequence} is a non-decreasing sequence of positive integers $a_1 \leq a_2 \leq a_3 \leq \dots$ such that every positive integer can be written as an index-distinct sum of some of the $a_k$ (in a not necessarily unique way). It was shown in \cite{Hons} and \cite{Brown} that a sequence is complete if and only if $a_1=1$ and $a_{n+1} \leq \sum_{k=1}^n a_k +1$ for all $n \geq 1$. 
Examples of complete sequences are: 
\begin{enumerate}
\item The sequence consisting of $1$ followed by the prime numbers listed in order of size. This follows by Bertrand's postulate that $p_{n+1} < 2p_n$ where $p_n$ is the nth prime. 
\item The powers of $2$, $a_n=2^{n-1}$, because every positive integer has a unique $2$-adic expansion. 
\item The Fibonacci numbers. Every positive integer can be written as a sum of distinct Fibonacci numbers by the so called Zeckendorf representation, and this representation is unique as long as consecutive Fibonacci numbers are never used.
\end{enumerate}

\begin{definition}
Fix $k$, a positive integer. An {\bf asymptotically $k$-complete sequence} is a non-decreasing sequence of positive integers $a_1 \leq a_2 \leq \dots$ with $d=gcd(a_1,a_2,\dots)$ such that there exists a positive integer $n_0$ so 
that for all integers $n \geq n_0$, we have that $nd$ is a $k$-fold sum of index-distinct $a_j$. Similarly, we will say a sequence is {\bf asymptotically $\leq k$-complete} if every large enough multiple of its gcd is a sum of $k$ or fewer index-distinct $a_j$.

A sequence of positive integers is {\bf asymptotically complete} if every large enough multiple of its gcd can be written as a sum of index-distinct $a_j$. We use the adjective ``weakly'' if any of the previous conditions hold after dropping the index-distinct requirement. 
\end{definition}

Note that if we have an asymptotically $k$-complete sequence with greatest common divisor $d$, we may divide all terms of the sequence by $d$ to get a new asymptotically $k$-complete sequence 
that represents every large enough integer as a sum of $k$ index-distinct terms of the sequence.  It also follows that scaling an asymptotically $k$-complete sequence by a positive integer 
yields another sequence of the same type. Also note that if a sequence starts with $1$, its gcd is automatically $1$.

The reader is warned that the definition of a complete integer sequence can be slightly different in different parts of the literature (some places have the word asymptotically implied and some don't - some don't build the gcd into the definition, but we do as it is more convenient for us to do so and seems more natural).

Examples: 
\begin{enumerate}[(i)]
\item The sequence of prime numbers ordered by size is asymptotically complete as every integer greater than 6 is a sum of distinct primes (See \cite{Riddell}).
\item Vinogradov~\cite{Vino} proved that every large enough odd positive integer is a sum of three primes (not necessarily distinct), and Helfgott~\cite{Helfgott} improved this to cover every odd integer $> 5$ and hence proved what 
was called the weak Goldbach conjecture. The actual Goldbach conjecture, which is still open, conjectures that every even integer $> 3$ is a sum of two (not necessarily distinct) primes. Thus, the sequence of prime numbers is conjecturally asymptotically weakly $\leq 3$-complete. 
\item Lagrange showed that the sequence of integer squares $1,4,9, \dots, n^2, \dots$ is weakly $\leq 4$-complete, i.e., every positive integer is a sum of 4 or fewer positive integer squares. Waring conjectured that for every $k \geq 2$, there is an $\alpha(k)$ such that the sequence of $k$th powers $1,2^k, 3^k, \dots$ is weakly $\leq \alpha(k)$-complete, i.e., that every positive integer is a sum of $\alpha(k)$ or fewer positive integer $k$th powers. Hilbert showed the existence of $\alpha(k)$ and its optimal/minimal value was determined by the further work of many mathematicians. 
\end{enumerate}

In this paper, we will study a certain probabilistic process that generates increasing sequences of positive integers, by generating a sequence $X_1, X_2, \dots$ of independent, identically distributed, 
positive integer valued, gap random variables and forming an increasing sequence of weights $W_1 < W_2 < \dots $ via $W_n = \sum_{k=1}^n X_k$. 

We show that in general, as long as the gap distribution has finite $\frac{1}{2}$-moment, that almost surely the resulting sequence of weights is asymptotically complete. 

We further show, given stronger conditions on the gap distribution, that with probability 1, the resulting sequence of weights is an asymptotically $k$-complete sequence for every fixed $k \geq 2$. 

We also prove a weak modular equidistributivity law for the resulting sequence of weights if the gcd is $1$. The weak modular equidistributivity states that if one fixes a modulus $M \geq 2$, then with probability 1, $$\lim_{N \to \infty} P(W_{k+N} = i \text{ mod } M | W_k = j \text{ mod } M ) = \frac{1}{M}$$ regardless of the value of $i, j \in \mathbb{Z}/M\mathbb{Z}$ and $k \in \mathbb{Z}$.

In topology, a {\bf quota or threshold complex} is a simplicial complex on a (potentially) countable infinite set of vertices $v_1, v_2, \dots $, where each vertex is given positive weight $W_1, W_2, \dots$ and a positive quota $q$ is prescribed. The quota complex $X(W_1, W_2, \dots ; q)$ is the simplicial complex whose faces are given by subsets of $\{v_1, v_2, \dots \}$ whose total weight sum 
is below the quota $q$. (Note that in the end, only the vertices whose weight is below $q$ are included in this complex). As discussed in \cite{PW}, the topology of these complexes is tightly connected to the  question of the distribution of the sums of the weights, and open conjectures about this topology/distribution are equivalent to the Goldbach conjecture, the Riemann hypothesis, and other open problems in number theory \cite{PW}.

In this context, our results prove that if one forms the quota complex $X(W_1, W_2, \dots ; q)$ with weights given by the randomly generated weights of our process and integer quota $q$, then 
with probability 1, for any $m \geq 1$ there is a $q_m$ such that for $q \geq q_m$, we have $H_m(X(W_1,W_2, \dots ; q), \mathbb{Z}) \neq 0$, i.e. asymptotically the $m$th integral homology of the quota complex is nonzero. This is essentially a consequence of the $k$-completeness of the weight sequence for all $k \geq 2$, see for example \cite{PW}.

To summarize, our main theorems are the following:

\begin{theorem}
\label{thm: premain} Let $W=\{W_1,W_2,\dots\}$ be a sequence of positive, integer-valued random variables (``weights'') such that the gaps $\{X_{i+1}=W_{i+1}-W_i\}_{i\in\mathbb{N}}$ are independent identically distributed positive, integer-valued random variables with finite $\frac{1}{2}$-moment. Then with probability 1, $W$ is asymptotically complete.
\end{theorem}

%
\begin{theorem}
\label{thm: mainUnBDD}
Let $W=\{W_1,W_2,\dots\}$ be a sequence of positive, integer-valued random variables (``weights'') such that the gaps $\{W_{i+1}-W_i\}_{i\in\mathbb{N}}$ are independent identically distributed positive, integer-valued random variables. Let $s_1, s_2, \dots$ and $p_1, p_2, \dots$ denote the possible values and probabilities for the gap distribution. Let $\frac{ - \log p_{*}}{s_*}=inf_{i \in \mathbb{N}} \frac{ - \log p_{i}}{s_i} \geq 0$. If the distribution of the gaps has a moment generating function which has a radius of convergence larger than  $\frac{ -2(\log{p_*})}{s_*}$, then for any fixed integer $m \geq 2$, with probability $1$ the sequence of weights is asymptotically $m$-complete. 
\end{theorem}

Note that Theorem~\ref{thm: mainUnBDD} is a significant strengthening of Theorem~\ref{thm: premain} as it is much easier for a set of positive integers to be asymptotically complete 
than to be asymptotically $m$-complete. For example, the set of prime numbers is asymptotically complete but only conjecturally asymptotically weakly $\leq 3$-complete.

Theorem~\ref{thm: mainUnBDD} applies to gaps with Poisson distribution and gaps with geometric distribution with $p > \frac{\sqrt{5}-1}{2}$. 

\begin{remark}
The conclusion of the theorem is a statement about $m$-fold sumsets. It is stating that for any integer
$m>1$, with probability $1$ the $m$-fold sumset $$\underbrace{W+\dots+W}_{m \text{ summands}}$$ contains all large enough integer multiples of the gcd of the gap sequence.
\end{remark}

For any $m\in\mathbb{N}$, if the gaps are allowed to be large with high enough probability, then the random gap sequence will not be asymptotically $m$-complete. In Section \ref{sec:counter}, we give a simple example demonstrating this that has finite $\alpha$-moment when $0 < \alpha < \frac{1}{m+1}$ but has infinite $\frac{1}{m+1}$-moment.

\section{Definitions and preliminary results}

\begin{definition} Given positive integers $s_1, \dots , s_k$, we will let $M(s_1,\dots , s_k)$ denote the monoid they generate, i.e.,
the set $$\left\{ \sum_{i=1}^k n_i s_i | n_i \text{ is a nonnegative integer } \right\}$$ under addition. 
\end{definition}

\begin{example}
The subgroup of the integers $(\mathbb{Z},+)$ generated by $5$ and $6$ is all of $\mathbb{Z}$, the set of multiples of $1=gcd(5,6)$, while the monoid 
generated by $5$ and $6$ is $M(5,6)=\{0, 5, 6, 10, 11, 12, 15, 16, 17, 18 \} \cup \{n \in \mathbb{Z} | n \geq 20 \}$.
\end{example}

Note that it is well-known that the group that $s_1, \dots, s_k$ generate is equal to the set of integer multiples of $gcd(s_1, \dots, s_k)$, their greatest common divisor. The elementary ``Stamp lemma'' states that the monoid generated by $s_1, \dots, s_k$, i.e., $\{ \sum_{i=1}^k a_i s_i | a_i \in \mathbb{Z}_{\geq 0} \}$, contains all large enough multiples of $gcd(s_1, \dots, s_k)$. We include a proof of this well-known fact in the appendix for completeness. 

Define the lower weight density of the set $S$ of positive integers to be $$\delta_-(S)=\liminf_{N \to \infty} \frac{|\{ s \in S | s \leq N \}|}{N}$$ and the upper weight density to be 
$$\delta_+(S)=\limsup_{N \to \infty} \frac{|\{ s \in S | s \leq N \}|}{N}.$$ Note that $0 \leq \delta_- \leq \delta_+ \leq 1$ in general. If $\delta_-(S)=\delta_+(S)$, we say the set $S$ has density $\delta(S)=\delta_-(S)=\delta_+(S)$.

If $1 \leq W_1 < W_2 < W_3 < \dots$ is an increasing sequence of positive integer weights, we will set $W=\{ W_i | i \in \mathbb{Z}_+ \}$. 
We let $W-W=\{ W_i - W_j | i, j \in \mathbb{Z}_+\}$ to be the ``difference set" of the weights and $W+W=\{ W_i + W_j | i,j \in \mathbb{Z}_+\}$ to be the ``sum set".
Finally, $W \oplus W=\{ W_i + W_j | i < j \in \mathbb{Z}_+ \}$ will denote the ``distinct weights sum set". We will be generating weights randomly, and so all of these sets and their densities will a priori be random in our process.

The fundamental model that we study in this paper is given as follows:

Let $X_1, X_2, X_3, \dots$ be a sequence of i.i.d. (independent and identically distributed) random variables which take only positive integer values. These random variables represent random gaps in an increasing sequence of weights $W_1 < W_2 < W_3 < \dots$ where $W_j=\sum_{i=1}^j X_i$. Thus, starting at 0, the $X_i$ are the consecutive gaps in the sequence of weights $0 < W_1 < W_2 < W_3 < \dots$.

Let the distribution that the $X_j$ share have the probability of value $i$ be denoted by $p_i$ for all $i \in \mathbb{Z}_+$. Let $X$ denote a prototypical random variable with the distribution given by 
this gap distribution.

We will write $Spt(X)=\{ i \in \mathbb{Z}_+ | p_i > 0 \}$ and say the gap distribution has finite support if $|Spt(X)| < \infty$.

We now prove some preliminary results: 

\begin{theorem}
\label{thm: prelim}
Let $X_1, X_2, X_3, \dots$ be a sequence of i.i.d. random positive, integer-valued ``gap'' variables. Let $W_j= \sum_{i=1}^j X_i$ be the corresponding sequence of weights and $W=\{ W_j | j \in \mathbb{Z}_+ \}$ be the set of all the weights.
Then, with probability 1:
\begin{enumerate}[(i)]
\item We have $$W-W = M(Spt(X)) \cup -M(Spt(X))$$ where $M(Spt(X))$ denotes the monoid generated by the support of $X$. 
\item If $E[X] < \infty$, we have $$\delta_-(W)=\delta_+(W) =\delta(W)=\frac{1}{E[X]}.$$ 
\end{enumerate}
\end{theorem}

\begin{proof}
Proof of (i): Let $x \in M(Spt(X))$. Thus, $x=\sum_{i=1}^n a_i s_i$, where $a_i$ are nonnegative integers and $p_{s_i} > 0$ for all $1 \leq i \leq n$. 
Note that we can think of the $X_i$ as the results of a sequence of independent trials. Consider the event that a sequence of trials results in outcome $s_1$ for the first $a_1$ trials, outcome $s_2$ for the next $a_2$ trials, followed by outcome $s_3$ for the next $a_3$ trials, etc., until the last $a_k$ trials with outcome $s_k$. The probability that a sequence of $a_1 + \dots + a_k$ consecutive trials results in this sequence of outcomes is $p_{s_1}^{a_1} \dots p_{s_k}^{a_k} > 0$. Thus, batching the independent trials into disjoint batches of length $a_1 + \dots + a_k$, we see that the probability that we never get a sequence of $a_1 + \dots + a_k$ consecutive trials with this pattern of results is zero. Thus, with probability 1, there exists $W_n$ such that $W_{n+a_1+\dots + a_k}=
W_n+(a_1s_1 + \dots +a_ks_k)=W_n+x$. Thus, $W-W$ contains both $x$ and $-x$. Since there are only a countable number of possible $x \in M(Spt(X))$, we may conclude that with probability 1, 
$W-W$ contains all of $M(Spt(X)) \cup -M(Spt(X))$.

Thus, $M(s_1, \dots, s_k) \cup -M(s_1, \dots, s_k) \subseteq W-W$. The converse set inclusion is trivial, as any $x=W_i-W_j\in W-W$ with $i \neq j$ is a finite integer linear combination of a subset of the numbers in $Spt(X)$ with all positive or all negative coefficients. The sign depends on whether $i<j$ or $i>j$.

Proof of (ii): This is a standard result in Renewal theory. See e.g. Theorem 2.4.6 in \cite{Durrett}.

\end{proof}

Note that the randomly generated set of weights $W$ is with probability 1 not the same in any two instances of the process (as long as $|Spt(X)| \geq 2$), but the set $W-W$ is with probability $1$ predetermined by the possible gap values, and hence not really random.

\section{Weak Modular Equidistributivity}

Let $X$ be a gap distribution from which we can generate a random sequence of integer weights $1 \leq W_1 < W_2 < W_3 < \dots$ as previously discussed. Further assume $gcd(Spt(X))=1$, and let 
$p_j=P(X=j)$. 

Given an integer modulus $M > 1$, this sequence of weights induces a Markov process in the cyclic group $\mathbb{Z}/M\mathbb{Z}$, which starts at $0$ and at each step 
adds value $x \in \mathbb{Z}/M\mathbb{Z}$ to the current value with probability $$\sum_{\{j | j = x \text{ mod } M \}} p_j.$$ Let $\mathbb{A}$ be the corresponding Markov matrix whose rows and columns are labelled by the elements of $\mathbb{Z}/M\mathbb{Z}$ in their standard order $0, 1, 2, \dots, M-1$ and whose $(i,j)$-entry is $$p_{ij}=\sum_{\{t | j=i+t \text{ mod } M\}} p_t.$$

Basically, this Markov chain keeps track of the mod $M$ congruence class of the weights we generated in the original process. 
The $(i,j)$-entry of $\mathbb{A}^N$ gives the conditional probability that $W_{k+N} = j \text{ mod } M$ given that $W_k=i \text{ mod } M$, for any positive integers $k, N, i, j$.

By the Stamp Lemma~\ref{lem: stamp lemma}, since $gcd(Spt(X))=1$,  $M(Spt(X))$, the monoid generated by the possible gap values,  contains all large enough integers. Thus, it follows that this Markov chain is regular (some power of $\mathbb{A}$ has all nonzero entries), and hence by the Frobenius-Perron theorem, 
$\mathbb{A}$ has a unique probability row eigenvector for eigenvalue $1$. If $\tau=(\tau_0, \dots, \tau_{m-1})$ denotes this eigenvector, then the theory of Markov chains tells us 
as $N \to \infty$, the rows of $\mathbb{A}^N$ all converge to $\tau$.
This means that $$\lim_{N \to \infty} P(W_{k+N}=j \text{ mod } M | W_k = i \text{ mod } M ) = \tau_j$$ for any positive integers $k, M, i, j$.

Furthermore, in our case, this Markov chain is based on a Cayley graph on $\mathbb{Z}/M\mathbb{Z}$ with probabilities given by $p_{ij}$ that are translation invariant i.e. $p_{i,j}=p_{i+T,j+T}$.
Due to this, it is easy to check that the rows of the matrix are related by the fact that each row is the cyclic shift of the row above it, one step to the right. In other words, $\mathbb{A}$ is a 
circulant matrix. Due to this, each entry of the first row occurs exactly once in each column also, and so all column and row sums equal to one (the matrix is doubly stochastic). 
This implies that $\tau=(\frac{1}{M}, \dots, \frac{1}{M})$.

Thus, our set of weights will be weakly equidistributed over any modulus in this case, i.e., 
$$\lim_{N \to \infty} P(W_{k+N}=j \text{ mod } M | W_k = i \text{ mod } M ) = \frac{1}{M}$$ for any positive integers $k, M$.

Intuitively, this means that in any modulus $M$, independent of what the current weight $W_i$ is in $\mathbb{Z}/M\mathbb{Z}$, a much later weight $W_{i+N}, N >> 0$, is equally likely to lie in 
any of the congruence classes modulo $M$.

\begin{example}
Let  $X$ take value $2$ with probability $p$ and $3$ with probability $1-p$, $0 < p < 1$, and set modulus $M=5$. The transition matrix for the induced regular Markov chain on $\mathbb{Z}/5\mathbb{Z}$ would be 
the circulant matrix 
$$
\mathbb{A}=\begin{bmatrix} 0 & 0 & p & 1-p & 0 \\ 0 & 0 & 0 & p & 1-p \\ 1-p & 0 & 0 & 0 & p \\ p & 1-p & 0 & 0 & 0 \\ 0 & p & 1-p & 0 & 0 \end{bmatrix}
$$
and the unique probability row vector corresponding to eigenvalue $1$ is $[\frac{1}{5}, \frac{1}{5}, \frac{1}{5}, \frac{1}{5}, \frac{1}{5}]$, 
implying the weak equidistributivity modulo $5$ of the corresponding weight sequences generated by this random process.
\end{example}

We summarize the discussion of this section in the following result:

\begin{theorem}[Weak Modular Equidistribution Theorem]
\label{thm: weakmodularequidistribution}

Let $\{X_i\}_i$ be a sequence of i.i.d. positive, integer-valued ``gap'' random variables, and let $W_1 < W_2 < W_3 < \dots $ be the sequence of weights these gaps generate, where $W_j = \sum_{i=1}^j X_i$. If $X$ refers to a prototypical gap random variable, further assume $gcd(Spt(X))=1$. 
Then with probability $1$, we have 
$$
\lim_{N \to \infty} P(W_{k+N}=j \text{ mod } M | W_k = i \text{ mod } M ) = \frac{1}{M}
$$
for all positive integers $k, i, j$, and $M$.
\end{theorem}

\section{Random weight sequences with finite 1/2-moment are always asymptotically complete}

In this section, we will prove that a random positive integer weight sequence is almost surely asymptotically complete as long as the underlying gap distribution has a finite $\frac{1}{2}$-moment.

First, a lemma:
\begin{lemma}
\label{lemma: squarerootlaw}
Let $X_1,X_2,\dots$ be a sequence of i.i.d. positive, integer-valued random variables whose common distribution has a finite $1/2$'th moment. I.e., $E\left(X_1^{\frac 12}\right)<\infty$. Then $$\lim_{n\rightarrow\infty}\frac{X_1+\dots+X_n}{n^2}=0,\text{ a.s.}$$
\end{lemma}
\begin{proof}
For any $i\in\mathbb{N}$, let $p_i=P(X_1=i)$, $Y_i=X_i/i^2$ and $Z_i=Y_i1_{Y_i\leq 1}$. We will apply Kolmogorov's three-series theorem to the sequence $Y_i$. First, let's show that for any $k\in\mathbb{N}$ we have $\sum_{i=1}^\infty E(Z_i^k)<\infty$. We have
\begin{align*}
\sum_{i=1}^\infty E(Z_i^k)
= &\sum_{i=1}^\infty E\left(\frac{X_i^k}{i^{2k}}1_{X_i\leq i^2}\right)
=\sum_{i=1}^\infty\sum_{j=1}^{i^2}\frac{j^k}{i^{2k}}p_j
=\sum_{j=1}^\infty\left(j^kp_j\sum_{i=\left\lceil \sqrt{j}\right\rceil}^{\infty}\frac{1}{i^{2k}}\right)
\\<&c\sum_{j=1}^\infty\left(j^kp_j\frac{1}{\left( \sqrt{j}\right)^{2k-1}}\right)
=c\sum_{j=1}^\infty\sqrt{j}p_j=cE\left(X_1^{\frac 12}\right)<\infty,
\end{align*}
where $c$ is some constant. By setting $k=1$ and $k=2$, we obtain $\sum_{i=1}^\infty E(Z_i)<\infty$ and for the variance $\sum_{i=1}^\infty V(Z_i)\leq\sum_{i=1}^\infty E(Z_i^2)<\infty$.

Similarly, we have
\begin{align*}
\sum_{i=1}^\infty P(Y_i>1)
= &\sum_{i=1}^\infty P(X_i>i^2)
=\sum_{i=1}^\infty\sum_{j=i^2+1}^{\infty}p_j
=\sum_{j=1}^\infty\sum_{i\leq\sqrt{j-1}}p_j
\\\leq&\sum_{j=1}^\infty \sqrt{j}p_j=E\left(X_1^{\frac 12}\right)<\infty.
\end{align*}
Thus, by Kolmogorov's three-series theorem, we have that $\sum_{i=1}^\infty X_i/i^2=\sum_{i=1}^\infty Y_i$ converges almost surely. It follows from Kronecker's lemma that $$\lim_{n\rightarrow\infty}\frac{X_1+\dots+X_n}{n^2}=0,\text{ a.s.}$$

\end{proof}

We now prove the main theorem of this section.

\begin{theorem} Let $\{X_i\}$ be an i.i.d. positive integer gap sequence with finite $\frac{1}{2}$-moment. Then with probability 1, the resulting weight sequence $W_n=X_1 + \dots + X_n$ is an asymptotically complete sequence. 

\end{theorem}
\begin{proof}
First of all, we may assume without loss of generality (by dividing out by the gcd if necessary) that $gcd(s_1, s_2, \dots)=1$, where the $s_j$ are the possible values of the distribution $X_1$.

Under this condition, our definition of asymptotically complete agrees with that in the book \cite{TV}. By lemma 12.16 in \cite{TV}, to show that the sequence $\{W_n\}$ is almost surely asymptotically complete, it is enough to show that almost surely, 
\begin{enumerate}
\item the finite distinct sums of weights intersect every infinite arithmetic progression and
\item the sequence $\{ W_n \}$ is subcomplete as defined in \cite{TV}, i.e., if the finite distinct sums from $\{ W_n \}$ contain an infinite arithmetic progression.
\end{enumerate}

By Theorem 12.17 in \cite{TV}, there exists a positive absolute constant $C >0$ such that an infinite set of positive integers $A$ is subcomplete if $|A \cap [1,n]| \geq C \sqrt{n}$ for all large enough $n$.
(This theorem is sharp in the sense that there exist examples of sets $A \subseteq \mathbb{Z}_+$ with $|A \cap [1,n]| = \Omega(\sqrt{n})$ which fail to be subcomplete.) 

By Lemma~\ref{lemma: squarerootlaw}, we have almost surely that $$\lim_{n\rightarrow\infty}\frac{X_1+\dots+X_n}{n^2}=0.$$

Setting $n=\lceil C\sqrt{m} \rceil$ it follows that for $m$ large enough, almost surely, $$X_1+ \dots + X_{\lceil C\sqrt{m} \rceil} \leq m$$ and so 
$|\{ W_n \} \cap [1,m]| \geq C \sqrt{m}$ for large enough $m$. Thus, almost surely by Theorem 12.17 in \cite{TV}, the weight sequence is subcomplete. 

To finish showing that the weight sequence is complete, it remains to show that with probability 1, it intersects every infinite arithmetic progression, as this implies in particular that its finite sums intersect the said progression. As there is a countable number of possible infinite arithmetic progressions, we may fix one, say $\{a, a+d, a+2d, \dots \}$ with $d \geq 1$, and just show that almost surely the weight sequence intersects it. Such an arithmetic progression consists of all integers equal to $a$ modulo $d$ after a given point and hence the weight sequence intersects it with probability 1, by 
the weak modular equidistribution theorem~\ref{thm: weakmodularequidistribution}. (One may take a large enough $N$ so that the transition probability to an element equal to $a$ modulo $d$ is at least  
$\frac{1}{2d}$, and then note by batching the sequence of weights after $a$ into independent batches of index length $N$, that there are infinitely many independent trials to achieve this transition and hence $0$ probability they all fail.)

\end{proof}

As we have seen in the introduction, asymptotic completeness is much weaker in general than asymptotic $2$-completeness. In later sections, we will achieve the much stronger result on asymptotic $2$-completeness under more stringent conditions on the gap distribution.

\section{Independence of intervals of gaps right before a stopping time}

Let $X_1, X_2, \dots$ be a sequence of i.i.d. positive, integer-valued random variables. The sample space for the set of possible sequences generated is the countable Cartesian product $Z=\mathbb{Z}_+ \times \mathbb{Z}_+ \times \cdots$. We equip it with the product topology coming from the discrete topology on the factors and the corresponding $\sigma$-algebra of Borel sets.
The Kolmogorov extension theorem guarantees the existance of a unique product Borel probability measure such that 
$$P(\{ \alpha_1\} \times \dots \times \{ \alpha_n \} \times \mathbb{Z}_+ \times \dots) = P(X_1=\alpha_1) \dots P(X_n=\alpha_n)$$ for all such product topology basis sets.

Note that a permutation of coordinates that fixes all but a finite set of coordinates gives a measure-preserving homeomorphism of the product space, as the coordinate measures are all equal as the random variables are identically distributed. More precisely, given a permutation $\sigma:=(\sigma_1,\dots,\sigma_m)$ of the integers $1,2,\dots,m$, let $S_{\sigma}$ be the map $S_{\sigma}:Z\rightarrow Z$ defined by $$S_{\sigma}(X_1,X_2,\dots)=(X_{\sigma_1},X_{\sigma_2},\dots,X_{\sigma_m},X_{m+1},X_{m+2},\dots).$$ For any finite permutation $\sigma$, we have that $S_{\sigma}$ is a measure preserving transformation of $Z$.

Now for any positive integer $n$, we can define the vaulting index $T_n: Z \to \mathbb{Z}_{\geq 0}$ by $T_n(\alpha_1, \dots ) = i-1$ where $i \in \mathbb{Z}_+$ is the unique index such that $\alpha_1 + \dots + \alpha_{i-1}< n$ and $\alpha_1 + \dots + \alpha_i \geq n$. It is easy to see that $T_n$ is Borel measurable and so is a well-defined random variable on our sample space.
 
Note that $T_n \leq n$ since the minimum possible gap value is $1$. Let us use the shorthand $X_{[a,b]} = \vec{v}$ to stand for the system of equations $$X_a=v(1), X_{a+1}=v(2), \dots, X_b=v(b-a+1)$$ for any $b-a+1$ dimensional vector $\vec{v}$ and positive integers $a, b$.

\begin{theorem}[Independence of prevault gap intervals]
\label{thm:endsIndependent}
Suppose there is a function $f(n)$ such that with probability 1, $f(n) \leq T_n$ for large enough $n$ with the property that $\lim_{n \to \infty} f(n) = \infty$.
If $a, b \in \mathbb{Z}_+$ with $a+b < f(n)$, then:
\begin{enumerate}
\item $P(X_{[T_n-(b-1),T_n]}=\vec{v}) = P(X_{[1,b]}=\vec{v})$ for all $b$-dimensional $\vec{v}$,

\item The random variables $X_1, \dots, X_a, X_{T_n},\dots, X_{T_n-(b-1)}$ are independent.
\end{enumerate}
\end{theorem}
\begin{proof}
Note that the assumption $a+b < f(n)$ guarantees the indices for the collection of random variables in the statement of the theorem are distinct. However, we must verify that the usage of the vault index $T_n$ in some of the indices did not introduce some subtle dependency amongst the collection of random variables which were a priori independent.

To prove (1), we use the partition of the sample space by the values of the vault index $T_n$:
$$P(X_{[T_n-(b-1),T_n]}=\vec{v}) = \sum_{\ell=f(n)}^{\infty} P(X_{[T_n-(b-1),T_n]}=\vec{v}, T_n=\ell-1).$$

Let $\sigma^l$ be the permutation $$\sigma^l:=(l-b,l-b+1,\dots,l-1,b+1,b+2,\dots,l-b-1,1,2,\dots,b),$$ i.e., the permutation that interchanges the block $[1,b]$ with the block $[l-b,l-1]$. Applying the measure preserving transformations $S_{\sigma^l}$ we obtain
\begin{align*}
P(X_{[T_n-(b-1),T_n]}=\vec{v})&=\sum_{\ell=f(n)}^{\infty} P(X_{[T_n-(b-1),T_n]}=\vec{v}, T_n=\ell-1)
\\&=\sum_{\ell=f(n)}^{\infty} P\left(S_{\sigma^l}(X_{[T_n-(b-1),T_n]}=\vec{v}, T_n=\ell-1)\right)
\\&=\sum_{\ell=f(n)}^{\infty} P(X_{[1,b]}=\vec{v}, T_n=\ell-1)
\\&=P(X_{[1,b]}=\vec{v}),
\end{align*}
as desired, where the second-to-last equality holds since the event $T_n=l-1$ is invariant under the transformation $S_{\sigma_l}$.

By (1), it follows that the collection of random variables $\{ X_{T_n}, \dots, X_{T_n-(b-1)} \}$ is independent since the collection $\{ X_1, \dots, X_b \}$ is independent. 
Thus to prove (2), it remains to show that 
$$
P(X_{[1,a]}=\vec{v}, X_{[T_n-(b-1), T_n]}=\vec{c}) = P(X_{[1,a]}=\vec{v}) P(X_{[1,b]}=\vec{c})
$$
for all $\vec{v}, \vec{c}$ vectors of suitable dimension.

We again use the partition by vault index $T_n=\ell-1$ to get:

$$
P(X_{[1,a]}=\vec{v}, X_{[T_n-(b-1), T_n]}=\vec{c})=\sum_{\ell=f(n)}^{\infty} P(X_{[1,a]}=\vec{v}, X_{[T_n-(b-1), T_n]}=\vec{c}, T_n=\ell-1).
$$

We now let $\pi^l$ be the permutation of the coordinates $1$ through $\ell-1$, which moves the last $b$ coordinates immediately after the first $a$-coordinates (shifting all other coordinates after the $a$th, by $b$ to the right):
$$\pi^l:=(1,\dots,a,l-b,l-b+1,\dots,l-1,a+b+1,a+b+2,\dots,l-b-1,a+1,a+2,\dots,a+b).$$
Since $S_{\pi^l}$ is a measure-preserving homeomorphism of the sample space, and since it also preserves the event $T_n=\ell-1$, we conclude the previous sum reduces to
\begin{align*}
\sum_{\ell=f(n)}^{\infty}& P(X_{[1,a]}=\vec{v}, X_{[T_n-(b-1), T_n]}=\vec{c}, T_n=\ell-1)\\&=
\sum_{\ell=f(n)}^{\infty} P\left(S_{\pi^l}(X_{[1,a]}=\vec{v}, X_{[T_n-(b-1), T_n]}=\vec{c}, T_n=\ell-1)\right)
\\&=\sum_{\ell=f(n)}^{\infty} P(X_{[1,a]}=\vec{v}, X_{[a+1, a+b]}=\vec{c}, T_n=\ell-1)
\\&=P(X_{[1,a]}=\vec{v}, X_{[a+1, a+b]}=\vec{c}),
\end{align*}
which completes the proof of (2) as we already know the variables $X_1, \dots, X_{a+b}$ are independent.

\end{proof}

\section{Proof of the m-completeness of random gap sequences}
%
In this section we give a proof of Theorem~\ref{thm: mainUnBDD}. 

Let the i.i.d. gap random variables $X_i$ take integer values $1 \leq s_1 < \dots < s_k < \dots $ with probability of taking value $s_j$ being $p_j > 0$, and 
generate the corresponding weights $W_n= \sum_{i=1}^{n} X_i$. We will further assume that $gcd(s_1, \dots, s_k,\dots )=1$ and that $p_2 > 0$, i.e., that there are at least two possible gap values. The general case follows immediately by dividing all the $s$'s by their greatest common divisor. Let $k\in\mathbb{Z}$ be such that $gcd(s_1,\dots,s_k)=1$. Such a finite index $k$ must necessarily exist.

As a first step, we obtain a bound on the maximum gap between the weights. 

\begin{lemma}
\label{lem:maxgap}
Let $X,X_1,X_2,\dots$ be i.i.d. random variables with moment generating function $M(t)=E[e^{tX}]$. Fix $a > 0$. If the radius of convergence of $M(t)$ is larger than $\frac 2a$, then with probability $1$ for any large enough $n$, we have $$\max(X_1,\dots,X_n)\leq a \log n.$$
\end{lemma}
\begin{proof}
Using Markov's inequality we have for any $t > 0$ such that $M(t)$ converges, 
\begin{equation*}
\varepsilon_n=P(X_n\geq a \log n)	= P(tX_n\geq at \log n)=P(e^{tX_n}\geq e^{at\log n}=n^{at})
\leq \frac{E[e^{tX_n}]}{n^{at}}.
\end{equation*}
Thus we have $\varepsilon_n \rightarrow 0$ as $n\rightarrow\infty$. 
For the maximum, we have
\begin{align*}
P(\max\{X_1,\dots,X_n\}\geq a\log n)&=1-(P(X_i<a\log n))^n
\\&=1-(1-P(X\geq a\log n))^n
\\&=1-(1-\varepsilon_n)^n
\\&=n\varepsilon_n+O((n\varepsilon_n)^2).
\end{align*}
Since the radius of convergence of $M(t)$ is larger than $\frac 2a$, there exists a constant $c>0$ and $t>\frac 2a$ such that $\varepsilon_n\leq c n^{-at}$. Since $t>\frac 2a$ implies $at-1>1$, it follows that 
$\sum_n P(\max\{X_1,\dots,X_n\}\geq a\log n)$ is finite. Thus, by the Borel-Cantelli lemma with probability one $P(\max\{X_1,\dots,X_n\}\geq a\log n)$ only occurs finitely many times. It follows that with probability 1 for any large enough integer $n$, we have $\max\{X_1,\dots,X_n\}\leq a\log n$
\end{proof}

The proof of the main theorem relies on the following lemma which bounds the probability that two independent, identically distributed gap processes as above do not intersect in the first $n$ steps.

\begin{lemma}
\label{lem:indepcopies}
Let $\{X_i\}_{i\in\mathbb{N}}$, $W=\{W_i\}_{i\in\mathbb{N}}$ be as before and let $\{Y_i\}_{i\in\mathbb{N}}$ be such that $\{X_i,Y_i\}_{i\in\mathbb{N}}$ are i.i.d. positive, integer-valued random variables with gcd of possible values equal to 1. Let $s_1, s_2, \dots$ and $p_1, p_2, \dots$ denote the possible values and probabilities for the gap distribution. For $k\in\mathbb{N}$, let $V_k=\sum_{i=1}^k Y_i$, and let $V=\{V_i\}_{i\in\mathbb{N}}$. In particular, we also assume that there exists $a > 0$ such that for all large enough $n$, 
$max(X_1,\dots X_n, Y_1, \dots Y_n) \leq a \log n$. Then there exists a positive constant $c$ such that for any $n\in\mathbb{N}$ we have
$$P\left(W\cap V\cap[1,n]=\emptyset\right) \leq 2e^{-c\frac{n^{1+a(\log p_{*})/s_*}}{\log n}},$$
where $(-\log p_{*})/s_*=\inf_{i\in\mathbb{N}} (-\log p_i)/s_i$ and the constant $c$ may depend on the gap distribution values and probabilities but nothing else. 
\end{lemma}
\begin{proof}
To simplify notation, throughout the the proof we will assume $p_1=p_{*}$, $s_1=s_*$. If $p_i=p_{*}$, then everywhere in the proof $s_1$ should be replaced by $s_i$ and $p_1$ by $p_i$. If the 
infimum is not achieved, we make the argument for each $i$ to get the result.

Since we have $\gcd(s_1, s_2,\dots )=1$ we also have $\gcd(s_1,\dots,s_k)=1$ for some $k$ and we may take $k$ large enough so that $k \geq *$. Then there exist positive integers $\alpha_2,\dots,\alpha_k$ such that $\sum_{i=2}^k\alpha_is_i\equiv 1\mod s_1$. This is because by the Stamp Lemma~\ref{lem: stamp lemma}, all large enough positive integer multiples of $gcd(s_2, \dots, s_k)$ are of the form $\sum_{i=2}^k \alpha_i s_i$ with $\alpha_i \in \mathbb{Z}_+$, and at least one of these multiples is congruent to $1$ mod $s_1$, as \\ $gcd(s_1, s_2,\dots,s_k)=1$. Let $\ell=\sum_{i=2}^k \alpha_i$ and let $I=\{i_1,\dots,i_{\ell s_1}\}$ be an integer sequence of length $\ell s_1$ and with period $s_1$, such that the first $\alpha_2$ terms are $2$, the next $\alpha_3$ terms are $3$, and so on. From the choice of the sequence $I$, it follows that for any integer $a$, we have 
\begin{equation}
\label{eq:saturatemods1}
\left\{\left(a+\sum_{q=1}^r s_{i_q}\right) \mod s_1\right\}_{1\leq r\leq \ell s_1}=\mathbb{Z}/s_1\mathbb{Z}.
\end{equation}
Let $\bar{S}=(s_{i_1},\dots,s_{i_{\ell s_1}})$ and define $\bar{p}=(p_2^{\alpha_2} \dots p_k^{\alpha_k})^{s_1} \in (0,1)$ to be the probability that $\ell s_1$ successive gaps are given by $\bar{S}$. Let $j=\left\lceil \frac{\sum_{q=1}^{\ell s_1} s_{i_q}}{s_1} \right\rceil + 1$, $\bar{j}=\lceil a\log n \rceil+js_1$ and $r=\lceil \frac{\bar{j}}{s_1} \rceil$. Let $\bar{q}=p_1^{r} > 0$ be the probability that $r$ successive gaps are equal to $s_1$. Note that $s_1 (j-1) \geq \sum_{q=1}^{\ell s_1} s_{i_q}$, and that $j$ and $\bar{p}$ only depend on the gap distribution and not on $n$.

These parameters are designed so that if there are $r$ successive positive integers with consecutive gaps equal to $s_1$ (we will call such a sequence an ``$s_1$-sieve''), then if we take the smallest weight achieved in the sieve (which is at most $a \log(n)$ from the start of the sieve by our assumption on gaps), then there is enough room in the sieve so that if the $\bar{S}$-pattern is achieved starting at this smallest weight, then this pattern will lie completely in the sieve.

We first make an argument that the $W$-weight sequence will achieve a reasonable number of independent (disjoint) $s_1$-sieves in $[1,n]$ with relatively high probability.

Divide the interval $[1,n]$ into sections of length $3\bar j$ each, where the $i$'th section is $((i-1)3\bar j,i3\bar j]$.  There will be $\tau=\lfloor \frac{n}{3\bar j}\rfloor$ sections. Now divide each of these sections into 3 subsections of equal length $\bar{j}$. Let $f_i$ be the index of the first weight $W$ in the $i$'th section, and let $l_i$ be the index of the first weight $W$ in the third (last) subsection of the $i$'th section. Note that almost surely there are such weights, since the largest gap is at most $a\log n<\bar j$ almost surely. Since the $f_i$'s and $l_i$'s are stopping times, by the strong Markov property we have that $X_{f_1},X_{f_1+1},\dots,X_{l_1},X_{f_2},X_{f_2+1},\dots,X_{l_2},\dots,X_{l_\tau}$ are independent and identically distributed with the same distribution as $X_1$. 

Define the $r$-tuples $\bar{X}_i$ as follows. If $l_i> f_i+r$, then $$\bar{X}_i=(X_{f_i},X_{f_i+1},\dots,X_{f_i+r-1}),$$ otherwise, $\bar{X}_i=(0,0,\dots,0)$. The sequence $\bar{X}_1,\dots,\bar{X}_\tau$ is i.i.d. From the definitions of the random sequences $\{l_i\}$ and $\{f_i\}$, it follows that for any index $i$, we have 
\begin{align*}
P(\bar{X}_i=(s_1,\dots,s_1))=&P(\bar{X}_i=(s_1,\dots,s_1)|l_i-f_i>r)P(l_i-f_i>r)\\&+P(\bar{X}_i=(s_1,\dots,s_1)|l_i-f_i\leq r)P(l_i-f_i\leq r)
\\=&P(\bar{X}_i=(s_1,\dots,s_1)|l_i-f_i>r)P(l_i-f_i>r)
\\=&P(\bar{X}_i=(s_1,\dots,s_1),l_i-f_i>r)
\\=&P(X_{f_i}=X_{f_i+1}=\dots=X_{f_i+r-1}=s_1)=p_1^r=\bar q.
\end{align*}

Thus, the expected number of occurrences of $s_1$-sieves $(\underbrace{s_1,\dots,s_1}_r)$ among the $\tau$ copies of $\bar{X}_i$'s is $\tau\bar q$ and these sieves are automatically disjoint and hence independent by construction.

It follows from Chernoff's inequality that the probability that the sequence $(\underbrace{s_1,\dots,s_1}_r)$ occurs at least $t:=\lceil \tau \bar{q}/2 \rceil$ times among the first $\tau$ of the $\bar{X}_i$'s is at least $1-e^{-\tau \bar q/8}$. Assuming this occurs, let $h_1,\dots,h_t\leq \tau$ be indices such that $\bar{X}_{h_i}=(s_1,\dots,s_1)$. Define $T_1,\dots,T_t$ by $T_i=\inf\{k | V_k \geq W_{f_{h_i}}\}$. 
These are hence independent stopping times in which the 2nd weight sequence enters these selected $s_1$-sieves of the first weight sequence.


Suppose $1\leq q\leq t$ is such that $(Y_{T_q+1},\dots,Y_{T_q+\ell s_1})=\bar{S}$. For a given $q$, the probability of this event is $\bar{p}$. Note also that when this happens, from the definitions of $T_q$, $\bar{S}$ and $\bar j$ we have $$V_{T_q+\ell s_1} = V_{T_q-1}+Y_{T_q}+ \sum_{q=1}^{\ell s_1} s_{i_q} \leq W_{f_{h_q}}+a\log n + (j-1)s_1 < W_{l_{h_q}}.$$ Thus, we have 
$$W_{f_{h_q}}\leq V_{T_q+1} \leq  V_{T_q+\ell s_1}\leq W_{l_{h_q}} \leq n.$$
Since for any $f_{h_q}\leq g< l_{h_q}$ the gap $W_{g+1}-W_g$ is $s_1$, all $W_g$ in this range give the same congruence class $\kappa$ mod $s_1$, and conversely every integer between $W_{f_{h_q}}$ and $W_{l_{h_j}}$ which is congruent to $\kappa$ mod $s_1$  is of the form $W_g$ for one of these indices.

It follows from \eqref{eq:saturatemods1}, and since the sequence $\{ V_{T_q}, \dots, V_{T_q+\ell s_1} \}$ has associated gap sequence $\bar{S}$, that there exists an index $T_q + 1 \leq g_0\leq T_q+\ell s_1$ such that $V_{g_0}\equiv \kappa \mod s_1$. It thus further follows that there exists an index $f_{h_q}\leq g_1< l_{h_q}$ such that $W_{g_1}=V_{g_0}$. This implies 
$W \cap V\cap[1,n] \neq \emptyset$. 

Since the $X$'s and $Y$'s are independent, the $T_i$'s are stopping times and the collections of random variables $\{Y_{T_i+1},\dots,Y_{T_i+\ell s_1}\}_{i\leq t}$ are independent, we obtain
$$P(W\cap V\cap[1,n]=\emptyset)<e^{-\tau\bar q/8}+(1-\bar{p})^t<2e^{-c_0 \tau\bar q},$$
where $c_0 > 0$ is a constant, and the last inequality holds since $\bar p$ is constant and $t=\lceil \tau \bar{q}/2 \rceil$. Using the definitions of $\tau$ and $\bar q$, we have 
$$\tau\bar q \geq c_1\frac{n}{\log n}p_1^{a\log n/s_1}=c_1\frac{n^{1+a\log p_1/s_1}}{\log n}$$ 
for some positive constant $c_1$. Setting $c=c_0c_1 > 0$ completes the proof of the lemma.

\end{proof}

From the proof of the lemma it follows easily that the following slight generalization also holds.

\begin{corollary}
\label{cor:shiftedIndepCopiesIntersect}
Let the random variables $X_i,Y_i,W_i,V_i$, $i\in\mathbb{N}$ be as in Lemma \ref{lem:indepcopies}. Then for any integers $a,b,\kappa=o(n)$ and any positive constant $d$, there exists a positive constant $c$ such that
$$P((a+W)\cap (b+V)\cap[\kappa,d n]=\emptyset)<2e^{-c\frac{n^{1+a(\log p_{*})/s_*}}{\log n}}.$$
\end{corollary}

We are now ready to prove the main theorem.

\begin{proof}[Proof of Theorem \ref{thm: mainUnBDD}]

Fix an integer $m \geq 2$.

For each $n \in \mathbb{Z}_+$, let $A_n$ denote the event that $n$ is not the sum of $m$ distinct weights in our process. 
If we can show that $\sum_{n=1}^{\infty} P(A_n) < \infty$ then the Borel-Cantelli lemma will imply that with probability $1$, only a finite number of the events $A_n$ occur, and hence that 
all sufficiently large integers $n$ must be a sum of $m$ distinct weights. Thus, we seek good upper bound estimates for $P(A_n)$. 

It follows from Lemma \ref{lem:maxgap} that almost surely $\max(X_1,\dots,X_n)\leq a \log n$ for any large enough $n$. Going forward, we will assume that indeed this is the case.

If $n$ is not the sum of $m$ weights, then for any $l$ we should have that $$n-(W_1+\dots+W_{m-2})-W_l$$ is not equal to a weight $W_r$. Given $n$, let $T_n$ be the index of the largest weight less than or equal to  $n-(W_1+\dots+W_{m-2})$:
$$T_n:=\max\{i:W_i\leq n-(W_1+\dots+W_{m-2})\},$$
and let $G_n$ be the gap:
$$G_n=n-(W_1+\dots+W_{m-2})-W_{T_n}.$$
By the definition of $T_n$, we have $0\leq G_n\leq a\log n$.

Consider the sequences $\{Q_i\}_{1\leq i< T_n}$ and $\{R_i\}_{1\leq i< T_n}$ defined by $$Q_i=X_{T_n}+X_{T_n-1}+\dots+X_{T_n-i+1}$$ and $$R_i=G_n+Q_i,$$ 
respectively. 
Using the definition of $G_n$ we have
\begin{align*}
R_i&=n-(W_1+\dots+W_{m-2})-W_{T_n}+X_{T_n}+X_{T_n-1}+\dots+X_{T_n-i+1}
\\&=n-(W_1+\dots+W_{m-2})-(W_{T_n}-X_{T_n}-X_{T_n-1}-\dots-X_{T_n-i+1})
\\&=n-(W_1+\dots+W_{m-2})-W_{T_n-i}.
\end{align*}
Suppose there are indices $r$ and $i$ such that $\frac n3\geq R_i=W_r>m a\log n$. Then we have 
$$n=W_1+\dots+W_{m-2}+W_{T_n-i}+W_r.$$
Since $W_r>ma\log n$, we have $r>m$. Moreover, $\frac n3\geq W_r$ and $\frac{n}{3} > ma \log n$ implies $$W_{T_n-i}\geq\frac{2n}3-(W_1+\dots+W_{m-2})\geq \frac{2n}3-(m-2)a\log n>W_r.$$ Thus, $m<r<T_n-i$ and we have that $n$ is the sum of $m$ distinct weights. It follows that 
\begin{equation*}
P(A_n)\leq P\left(W\cap \{R_i\}_{1\leq i < T_n}\cap\left(ma\log n,\frac n3\right]=\emptyset\right).
\end{equation*}
Since $G_n\leq a\log n$, we obtain 
\begin{equation}
\label{eq:AnVSdistinctEnds}
P(A_n)\leq \sum_{b=1}^{a\log n}P\left(W\cap \{b+Q_i\}_{1\leq i < T_n}\cap\left(ma\log n,\frac n3\right]=\emptyset\right).
\end{equation}

It follows from Theorem \ref{thm:endsIndependent} that the process $\{Q_i\}_i\cap \left[0,\frac n3\right]$ is an independent copy of the process $\{W_i\}_i\cap \left[0,\frac n3\right]$. 
Thus, Corollary \ref{cor:shiftedIndepCopiesIntersect} applies, and using it for each summand in \eqref{eq:AnVSdistinctEnds}, we obtain that that there exists a positive constant $c$ such that 
$$P(A_n)\leq 2a\log n e^{-c\frac{n^{1+a(\log p_{*})/s_*}}{\log n}}.$$

Thus, as long as $1+a (\log p_{*})/s_* >0$, we have $\sum_{n=1}^\infty P(A_n)<\infty$, so by the Borel-Cantelli lemma, with probability $1$, only finitely many positive integers $n$ cannot be written as a sum of $m$ distinct weights.

This requires $a < \frac{s_*}{-\log(p_*)}$, and hence that the gap moment generating function $M(t)$ has radius of convergence larger than $\frac{-2\log(p_*)}{s_*}$ in order to use Lemma~\ref{lem:maxgap}.
\end{proof}

Note that this last theorem applies to the situation where the gaps have (shifted) Poisson distribution i.e. $1+X$ with $X\sim\text{Poisson}(\lambda)$, since in that case, the moment generating function has infinite radius of convergence.

In the case of gaps with geometric distribution with parameter $p > 0$, we have $s_k=k$ and $p_k=(1-p)^{k-1}p$ for all $k \geq 1$. It is then easy to compute that 
$$-2\log{p_*}/s_*= max(-2\log(p),-2\log(1-p)),$$ and that the radius of convergence of $M(t)$ is $-\log(1-p)$. Thus, in order for the theorem to apply, we need $p \geq 0.5$ and $-\log(1-p) > -2\log(p)$ which translates to $p^2 + p - 1 >0$ which holds exactly when $p > \frac{-1+\sqrt{5}}{2}$.

\section{A counterexample}
\label{sec:counter}
For any natural number $m$, allowing the gaps to be large with high enough probability will make it impossible to write every large enough integer as the sum of $m$ weights. For example, let the gap distribution be $P(X=2^k)=p^k(1-p)$ for $k\geq 0$ with $p=2^{-\frac{1}{m+1}}$. Then we have 
$$P(X>n^{m+1})=P(X>2^\frac{(m+1)\ln n}{\ln 2})\approx p^\frac{(m+1)\ln n}{\ln 2}=\frac{1}{n}.$$
Thus, $\sum_n P(X_n>n^{m+1})=\infty$. Since the gaps $X_1,X_2,\dots$ are independent, by the second Borel-Cantelli lemma, we have that with probability 1, $X_n>n^{m+1}$ occurs infinitely often. Let $n$ be such that $X_n>n^{m+1}$. Since $W_n \geq X_n$, we have that $W_n>n^{m+1}$. Thus, for any integer $k<n^{m+1}$, if $k$ is the sum of $m$ weights $W_{i_1}+\dots +W_{i_m}$, then we must have $i_1,\dots,i_m<n$. However, there are at most $n^m$ distinct integers among the numbers $W_{i_1}+\dots +W_{i_m}$ with $i_1,\dots,i_m<n$, thus at most $n^m$ of the integers less than $n^{m+1}$ can be written as the sum of $m$ weights. Thus, there is an integer $\geq$ $n^{m+1}-n^{m}-1$ which cannot be written as a sum of $m$ distinct weights. As this holds for arbitrarily large $n$, we are done.

Note that in this counterexample, the gap distribution has finite $\alpha$-moment for any $0 < \alpha < \frac{1}{m+1}$ but the $\frac{1}{m+1}$-th moment is infinite.

\section{Appendix}

The following stamp lemma shows that the monoid generated by the $s_1, \dots, s_k$ coincides with the group generated by them for large enough integers. 
It is called the stamp lemma, as the set of postage values that  one could make using only stamps of valuations $s_1, \dots, s_k$ is exactly the monoid they generate, and 
so this lemma characterizes all large enough postages that can be made.

\begin{lemma}[The Stamp Lemma]
\label{lem: stamp lemma}
Let $1 \leq s_1 < s_2 < s_3 < \dots < s_k$ be a finite sequence of positive integers with $d=gcd(s_1,\dots, s_k)$, then 
there exists $n_0 \in \mathbb{Z}_+$ such that $M(s_1, \dots, s_k)$ contains all $nd$ for $n \geq n_0$. Furthermore, there is also an $n_0'$ such that for every $n \geq n_0'$, $nd=\sum_{i=1}^k \alpha_i s_i$ for some positive integers $\alpha_i$.
\end{lemma}
\begin{proof}
Let $\langle s_1,\dots,s_k\rangle$ be the subgroup of $(\mathbb{Z},+)$ that the sequence generates. Then, as every subgroup of $Z$ is cyclic, 
$\langle s_1, \dots, s_k\rangle = \langle D\rangle $ for some $D > 0$. As $s_j \in \langle D\rangle $, $D$ is a common divisor of all the $s_j$. On the other hand, as 
$D \in \langle s_1, \dots, s_k\rangle $, we have integers $m_i$ such that $D = \sum_{i=1}^k m_i s_i$. From this last equation (Bezout's equation), we see that any common divisor 
of the $s_i$ must divide $D$ and hence $D=d=gcd(s_1, \dots, s_k)$.

Dividing Bezout's equation by $D=d$, we see that $1=\sum_{i=1}^k m_i \frac{s_i}{d}$ from which it follows that $\langle \frac{s_1}{d}, \dots, \frac{s_k}{d}\rangle =\langle 1\rangle $, and so to prove the lemma, it is sufficient 
to show that there is an $n_0$ such that for all $n \geq n_0$, $n \in M(\frac{s_1}{d}, \dots, \frac{s_k}{d})$, i.e., it is enough to prove the lemma in the case where the $s_i$ are relatively prime, and so 
we assume that for the remainder of the proof.

Now, as $s_1$ is the smallest of the $s_j$, it will follow that $M(s_1, \dots, s_k)$ has all sufficiently large integers once we show that $s_1$ consecutive integers are in this monoid. (As all further integers can then be obtained from these by repeatedly adding $s_1$). Also note that if $s_1=1$ the result is trivial so WLOG $1 < s_1$.

As $s_1, s_k$ are relatively prime, by Bezout, there exist integers $a_1,\dots,a_k$ such that $1 = \sum_{i=1}^ka_is_i$. Let $I_1:=\{i:a_i<0\}$ and $I_2:=\{i:a_i\geq 0\}$, and let $b_i:=-a_i> 0$ for all $i\in I_1$. We have 
\begin{equation}
\label{eq:stamp_replace}
\sum_{i\in I_2}a_is_i=1+\sum_{i\in I_1}b_is_i.
\end{equation}

Let $n_0=\sum_{i\in I_1}s_1b_is_i$. Since $b_i>0$ for all $i\in I_1$, we have $n_0\in M(s_1,\dots,s_k)$. It follows from \eqref{eq:stamp_replace} that swapping $b_i$ copies of $s_i$ for all $i\in I_1$ by $a_i$ copies of $s_i$ for all $i\in I_2$ will increase the sum by $1$. Since the coefficients of $s_i$ for $i\in I_1$ are $s_1b_i$, and since $a_i\geq 0$ for all $i\in I_2$, we can make $s_1$ such swaps and after each swap still get an integer from the monoid $M(s_1,\dots,s_k)$. Since each swap increases the number by $1$, we obtain that $n_0,n_0+1,\dots,n_0+s_1$ are all in $M(s_1, \dots, s_k)$. It follows from our previous comments that $M(s_1, \dots, s_k)$ contains all large positive integer multiples of $d$.

Finally, if every $n \geq n_0$ has $nd=\sum_{i=1}^k f_i s_i$ for $f_i$ nonnegative integers, we may set $s_i=t_id$ for $1 \leq i \leq k$ and note that by adding $s_1 + \dots + s_k$ to both sides of the last sum, we get that for $n \geq n_0 + \sum_{i=1}^k t_i$,  $nd=\sum_{i=1}^k(f_i+1)s_i$ is a linear combination of the $s_i$ with positive integer coefficients.

\end{proof}


\begin{thebibliography}{20}

\bibitem{Brown} J.L. Brown, \emph{Note on Complete Sequences of Integers}, The American Math. Monthly 68 (6) (1961), 557-560.

\bibitem{Durrett} R. Durrett, \emph{Probability: Theory and examples, 4th edition}, Cambridge University Press, 2010.

\bibitem{Helfgott} H. Helfgott, \emph{The ternary Goldbach conjecture is true}, (2013), arXiv: 1312.7748 

\bibitem{Hons} R. Honsberger, \emph{Mathematical Gems III}, Math. Assoc. Amer., (1985), 123-128.

\bibitem{PW} J. Pakianathan and T. Winfrey, \emph{Threshold complexes and connections to number theory}, Turk. J. Math., {\bf 37}, Issue 3, (2013), 511-539. 

\bibitem{Riddell} J. Riddell, \emph{Partitions into distinct small primes}, Acta Arithmetica {\bf 41}, (1982), 71-84.

\bibitem{TV} T. Tao and V. Vu, \emph{Additive Combinatorics}, Cambridge University Press, 2010.

\bibitem{Vino} I.M. Vinogradov, \emph{}, Dokl. Akad. Nauk SSSR {\bf 15} (1937), 291-294.


\vskip.125in

\end{thebibliography}
\end{document}